\documentclass[10pt]{amsart}

\usepackage{setspace} 
\usepackage{amssymb}
\usepackage{amsmath}
\usepackage{amsthm}
\usepackage[a4paper]{geometry}
\usepackage{amsfonts}
\usepackage{bbm}
\usepackage{xcolor}
\usepackage{changes}

\usepackage{color}
\usepackage[colorlinks=true,linkcolor=blue,citecolor=blue,pdfpagelabels=false]{hyperref}
\usepackage{cleveref}
\usepackage{url}

\usepackage{hyperref}

\usepackage{cancel} 

\allowdisplaybreaks[4] 
\newtheoremstyle{myremark}     {10pt}{10pt}{}{}{\bfseries}{.}{.5em}{}

\newtheorem{thm}{Theorem}[section]

\newtheorem{lem}[thm]{Lemma}
\newtheorem{prop}[thm]{Proposition}
\theoremstyle{definition}
\newtheorem{defn}[thm]{Definition}

\theoremstyle{myremark}
\newtheorem{rem}[thm]{Remark}
\numberwithin{equation}{section}

\def\C{\mathbb C}

\def\N{\mathbb N}

\def\R{\mathbb R}

\newcommand{\CB}{\mathcal{B}}

\newcommand{\CD}{\mathcal{D}}
\newcommand{\CE}{\mathcal{E}}

\newcommand{\h}{\mathcal{H}}

\newcommand{\CM}{\mathcal{M}}
\newcommand{\CN}{\mathcal{N}}

\newcommand{\CP}{\mathcal{P}}

\newcommand{\abs}[1]{\left\vert#1\right\vert}

\newcommand{\norm}[1]{\left\Vert#1\right\Vert}

\newcommand{\ri}{\rightarrow}
\newcommand{\limn}{ \underset{n \ri \infty}{\lim} }
\newcommand{\auleq}{\underset{a.u.}{\leq }}

\begin{document}

	\title[Law of iterated logarithm]{Non-commutative Law of iterated logarithm}

	\author[S. Panja]{Sourav Panja}
	\address{Theoretical Statistics and Mathematics Unit, Indian Statistical Institute, Delhi, India}
	\email{souravpanja07@gmail.com}

	\author[\'E. Ricard]{\'Eric Ricard}
	\address{Université Caen Normandie, CNRS, LMNO UMR6139, F-14000 CAEN, FRANCE.}
	\email{eric.ricard@unicaen.fr}
	
	\author[D. Saha]{Diptesh Saha}
	\address{Theoretical Statistics and Mathematics Unit, Indian Statistical Institute, Delhi, India}
	\email{dptshs@gmail.com}

	
	

    \begin{abstract}
    We prove optimal non-commutative analogues of the classical Law of
      Iterated Logarithm (LIL) for both martingales and sequences of
      independent (non-commutative) random variables. The classical
      martingale version was established by Stout
      \cite{stout1970martingale} and the independent case by
      Hartman–Wintner \cite{hartman1941law}. Our approach relies on a
      key exponential inequality essentially due to Randrianantoanina
      \cite{randrianantoanina2024triple} that improves that from Junge
      and Zeng \cite{junge2015noncommutative}. It allows to derive an optimal
      non-commutative Stout-type LIL just as in \cite{zeng2015kolmogorov}; from that martingale result we
      then deduce a non-commutative Hartman–Wintner-type LIL for
      independent sequences of random variables.
    \end{abstract}

\maketitle
\section{Introduction}

Within probability theory, the law of iterated logarithm (LIL) occupies a unique position among limit theorems: while the strong law of large numbers (LLN) describes the almost sure convergence of normalized sums to a constant, and the central limit theorem (CLT) characterizes the limiting distribution of these sums after normalization, the LIL precisely identifies the scale of the fluctuations that remain between these two extremes. The earliest contributions in the context of independent increments are due to Khintchine, Kolmogorov, and Hartman--Wintner; a historical overview can be found in~\cite{Bauer96Prob}. Stout subsequently extended the LIL of Kolmogorov and Hartman--Wintner to the martingale setting in ~\cite{stout1970martingale,stout1970hartman}.  

Beyond the scalar case, analogues of the LIL have been established for sums of independent random variables taking values in Banach spaces, due to Kuelbs, Ledoux, Talagrand, Pisier, and others (see~\cite{ledoux1991probability} for detailed references).
In recent decades, non-commutative probability theory has emerged as a powerful generalization of classical probability, motivated by the study of operator algebras, quantum statistical mechanics, and free probability. In this setting, the role of the probability space is played by a von Neumann algebra $\mathcal{M}$ equipped with a faithful, normal, tracial state $\tau$, and random variables correspond to $\tau$-measurable operators affiliated with $\mathcal{M}$. In this framework, convergence theorems for martingales and ergodic theorems have been studied extensively (cf. \cite{cuculescu1971martingales},  \cite{junge2002doob}, \cite{lance1976ergodic}, \cite{JX07} and references therein). On the other hand, both the LLN and the CLT have also received considerable attention in the literature; see, for example, \cite{batty1979strong, Luc, Quang, Voi} and the references cited there. However, the progress on the LIL has been more recent and relatively limited in the non-commutative setting. In particular, Konwerska~\cite{konwerska2008law} obtained a version of the Hartman--Wintner type LIL for sequences of independent self-adjoint operators and Zeng~\cite{zeng2015kolmogorov} an analogue of Stout's result that we both aim to improve. 

The martingale version of Kolmogorov's LIL \cite{kolmogoroff1929gesetz} is due to Stout~\cite{stout1970martingale}. Let us first recall Stout's LIL. Let $(X_n, \mathcal{F}_n)_{n \geq 1}$ be a martingale with $\mathbb{E}[X_n] = 0$ and define the martingale differences $Y_n = X_n - X_{n-1}$ for $n \geq 1$, with $X_0 = 0$. Set
\[
s_n^2 = \sum_{i=1}^n \mathbb{E}[Y_i^2 \mid \mathcal{F}_{i-1}]
\] and for $x>1$, set
$
L(x) = \max\{1, \ln \ln x\}.
$
If $s_n^2 \to \infty$ a.s. and if for some sequence $(\alpha_n)$ of positive reals with $\alpha_n \to 0$,
\begin{equation}
|Y_n| \le \frac{\alpha_n s_n}{\sqrt{L(s_n^2)}} \quad \text{a.s.},
\label{eq:KolmogorovCondition}
\end{equation}
then
\[
\limsup_{n \to \infty} \frac{X_n}{\sqrt{s_n^2 L(s_n^2)}} = \sqrt{2} \quad \text{a.s.}
\]

In \cite{zeng2015kolmogorov}, the author extended  Stout's result to non-commutative setup. In this setting, we assume that $\CM$ is a finite von Neumann algebra with a faithful, normal (f.n. henceforth) tracial state $\tau$. Furthermore, let $(\CM_k)_{k \in \N}$ be an increasing sequence of von Neumann subalgebras with the trace preserving conditional expectations $(E_k)_{k \in \N}$, where $E_k: \CM \to \CM_k$ with the property $E_n \circ E_m= E_{\min\{m,n\}}$. Let $(x_k)_{k \in \N}$ be a sequence of non-commutative martingales in $\CM$ and define $x_0=0$, and let $d_k:= x_k - x_{k-1}$ be the associated martingale differences. Further, for all $n \in \N$ define,
\[
s_n^2 = \left\| \sum_{k=1}^n E_{k-1}(d_k^2) \right\|_\infty, \quad u_n = \sqrt{L(s_n^2)}.
\]
Then, Zeng proved the following theorem. For other relevant definitions we refer to \Cref{sec: prelim}.
\begin{thm}\label{thm: zeng}
Let $(x_n)_{n \in \N \cup {0}}$ be a sequence of self-adjoint martingales in $\CM$ with $x_0=0$. Further suppose that $s_n^2 \to \infty$ and $\|d_n\|_\infty \le \alpha_n s_n / u_n$ for some sequence $(\alpha_n)$ of positive numbers with $\alpha_n \to 0$. Then
\[
\limsup_{n \to \infty} \frac{x_n}{s_n u_n} \underset{a.u.}{\leq } 2.
\]
\end{thm}

First, we observe that in this result the bound is $2$ rather than
$\sqrt{2}$, which seems inconsistent with the classical case. To
establish \Cref{thm: zeng}, the author invokes an exponential
inequality from \cite[Lemma~2.2]{junge2015noncommutative}. However,
its application to \cite[Theorem~1]{zeng2015kolmogorov} is
slightly incorrect. To address this issue, we begin by
constructing an appropriate sequence in \Cref{prop: seq lemma}. We
then refine the exponential inequality using results by
Randrianantoanina \cite{randrianantoanina2024triple}, and by employing
a standard technique we extend \Cref{thm: zeng} to the
non-self-adjoint martingale case. This yields a version consistent with the classical setting. Our main result is stated
below.

\begin{thm}
    Let $(x_n)_{n \geq 0}$ be a non-commutative martingale in $(\mathcal{M}, \tau)$ with $x_0 = 0$. Define 
    $$
    s_n^2=\max\left\{\norm{\sum_{i=1}^n E_{i-1}(d_i^*d_i)}_\infty,\norm{\sum_{i=1}^n E_{i-1}(d_id_i^*)}_\infty\right\} \text{ and } u_n=L(s_n^2)^{1/2}.
    $$
    Suppose that $s_n^2 \to \infty$ and $\norm{d_n}_\infty \leq \frac{\alpha_n s_n}{u_n}$, where $(\alpha_n)$ is a sequence of positive numbers such that $\alpha_n \to 0$, then 
    $$
    \limsup_{n \to \infty} \frac{x_n}{s_n u_n} \underset{a.u.}{\leq }\sqrt{2}.
    $$
  \end{thm}

As already pointed out in \cite{zeng2015kolmogorov}, one cannot have equality for general non-commutative random variables because of the free CLT \cite{Voi}. In some sense, the commutative situation is the worst possible for the LIL.

Next, we use this martingale analogue of LIL to deduce a non-commutative analogue of independent random variable version due to Hartman and Wintner \cite{hartman1941law}. They stated that if $(y_i)$ is an i.i.d. sequence of real, centered, square-integrable random variables with variance $\sigma^2$, then 
\begin{align*}
    \limsup_{n \to \infty} \frac{\sum_1^n y_i}{\sqrt{n L(n)}}= \sqrt 2 \sigma \quad a.s.
\end{align*}
A first non-commutative analogue of this result was  established in \cite[Theorem 3.23]{konwerska2008law}, where the author obtained a similar conclusion without really relying on martingales.  We emphasize that the result there concerns bounds in the bilaterally almost uniform (b.a.u.) sense, whereas \cite[Theorem 1]{zeng2015kolmogorov} provides bounds in the almost uniform (a.u.) sense. More recently, \cite{hong2024failure} exhibited examples of non-commutative probability spaces $(\CM, \tau)$ and martingale sequences in $L_p(\CM, \tau)$ for $1 \leq p <2$ that fail to converge almost uniformly. This makes it natural to investigate a.u. bounds for the limits of such sequences. In this article, our approach (cf. \Cref{thm: main them with independent rv}), in the spirit of \cite{hartman1941law} or \cite{deAcosta},   employs martingale exponential inequalities, allowing us to derive a sharper bound in the almost uniform sense rather than in the bilaterally almost uniform one.

The structure of the paper is as follows. Section~\ref{sec: prelim} reviews the necessary background on non-commutative $L_p$-spaces, conditional expectations, and boundedness concepts used in our analysis. Section~\ref{sec: main sec} presents the detailed proof of the main theorem. We conclude with an extension of our main result as in \cite{stout1970martingale}. Next, in Section~\ref{sec: lil for indep rv}, we establish the non-commutative LIL for independent random variables.

\section{Preliminaries}\label{sec: prelim}
Throughout this section, we assume that $\CM (\subseteq \CB(\h))$ is a finite von Neumann algebra with f.n. tracial state $\tau$, where $\h$ is assumed to be a separable Hilbert space. A closed, densely defined operator $x: \CD(x) \subseteq \h \to \h$ is called affiliated to $\CM$ if $xu \subset ux$ for all unitaries in $\CM'$, the commutant of $\CM$. An operator $x$, which is affiliated to $\CM$ is called $\tau$-measurable if for all $\epsilon>0$ there exists a projection $e$ in $\CM$ such that $\tau(1-e)< \epsilon$ and $\norm{xe}< \infty$. Since $\tau$ is finite, it is observed that every closed, densely defined operator that is affiliated to $\CM$ is $\tau$-measurable. The $*$-algebra of closed, densely defined operator, which are affiliated to $\CM$ is denoted by $\widetilde{\CM}$. Throughout this article, the set of all projections will be denoted by $\CP(\CM)$.

\begin{defn}\label{defn: auleq}
Let $(x_n)_{n \in \N}$ be a sequence in $\widetilde{\CM}$ and $K \geq 0$. We say $\limsup_{n \to \infty} x_n \auleq K$, that is,  $(x_n)_n$ is almost uniformly (a.u.) bounded by $K$ if for all $\epsilon, \delta>0$, there exists $p \in \CP(\CM)$ such that $\tau(1-p)< \epsilon$ and 
\[
\limsup_{n \to \infty} \norm{x_np} \leq K + \delta.
\]
\end{defn}

\begin{defn}\label{defn: auconv}
Let $(x_n)_{n \in \N}$ be a sequence in $\widetilde{\CM}$ and $x \in \widetilde{\CM}$. We say  $(x_n)_n$ is almost uniformly (a.u.) convergent to $x$ if for all $\epsilon>0$ there exists $p \in \CP(\CM)$ such that $\tau(1-p)< \epsilon$ and 
\[
\limn \norm{(x_n-x)p}=0.
\]
\end{defn}

Our next remark concludes that the definition of the almost uniform boundedness and the almost uniform convergence is independent of the ambient von Neumann algebra. 

\begin{rem}\label{rem: indep of ambient vna}
Let $\CN$ be a von Neumann subalgebra of $\CM$. Let $(x_n)$ be a sequence in $\CN$ with $\limsup_{n \to \infty} x_n \auleq K $ in $\CM$ for some constant $K>0$. Let $\epsilon, \delta>0$. We first choose $0< \eta <1$ such that $(K+ \frac{\delta}{2})/ (1-\eta) \leq K+ \delta$. Then, there exists $p \in \CP(\CM)$ such that $\tau(1-p)< \eta \epsilon$ and 
\[
\limsup_{n \to \infty} \norm{x_np} \leq K + \frac{\delta}{2}.
\]
We claim that there exists $q \in \CP(\CN)$ such that $\tau(1-q)< \epsilon$ and 
\[
\limsup_{n \to \infty} \norm{x_nq} \leq K + \delta.
\]

Consider $p':= E(p)$, where $E$ is the faithful, normal conditional expectation from $\CM$ onto $\CN$. Now, consider $q:= 1_{[1-\eta, 1]}(p')\in \CN$. Next, observe that $1- p' \geq \eta (1-q)$ and
\[
\eta \epsilon \geq \tau(1-p)= \tau(1-p') \geq \eta \tau(1-q).
\]
On the other hand, since $E(p)\geq (1-\eta) q$, we have \[
(1-\eta)\norm{x_nq}^2=(1-\eta)\norm{x_nqx_n^*} \leq  \norm{x_n E(p) x_n^*}=\norm{E(x_n p x_n^*)} \leq  \norm{x_np}^2.
\]
Therefore, we obtain $\tau(1-q) \leq \epsilon$ and
\[
\limsup_{n \to \infty} \norm{x_nq} \leq \frac{1}{1-\eta} \limsup_{n \to \infty} \norm{x_np} \leq K+ \delta.
\]
A similar argument also holds true for almost uniform convergence.
\end{rem}

We state the following lemma, which will be used repeatedly in the subsequent results. However, the proof is straightforward and is left to the reader.

\begin{lem}\label{lem: au convergence and bounded lemma}
    Let $(x_n)$ and $(y_n)$ be two sequences in a von Neumann algebra $\CM$ such that $\limsup\limits_{n \to \infty} x_n \auleq K$ and $y_n\xrightarrow{a.u} 0$ for some constant $K>0$, then $$\limsup_{n \to \infty} \left(x_n+y_n\right) \auleq K.$$
\end{lem}

 We will need a version of the Kronecker lemma for the a.u. convergence.
   \begin{lem}\label{lem: Kronecker2}
    Let $(x_n)$ be a sequence in $\widetilde{\CM}$, $(\alpha_n)_n$ an increasing sequence of positive real numbers, which diverges to $+\infty$. Then
    $$
    \sum_{i=1}^n \frac{x_i}{\alpha_i} \text{ converges a.u. }  
    \quad \Rightarrow \quad 
     \frac{1}{\alpha_n} \sum_{i=1}^n x_i  \text{ converges a.u. to }   0.
     $$
       \end{lem}

   \begin{proof}
Let $S_n=\sum_{k=1}^n \frac {x_k}{\alpha_k}$ and set $\beta_k=\alpha_k-\alpha_{k-1}$ ($\alpha_{-1}=0)$. Then $\sum_{k=1}^n \beta_k=\alpha_n\to \infty$. We denote $y_n=\frac{\sum_{k=1}^n \beta_k S_k}{\alpha_n}$.

     Using an Abel transform, we have
     $\frac 1 {\alpha_{n+1}}\sum_{k=1}^{n+1} x_k=S_{n+1}+\frac{\alpha_1S_1}
     {\alpha_{n+1}}-y_n$. As $x_1\in \widetilde{\CM}$, $\frac{\alpha_1S_1}
     {\alpha_{n+1}}\to_{a.u.}0$, thus we only need to show that $\lim_{a.u.}y_n=\lim_{a.u.}S_n=S$.

     Let $\epsilon>0$, then there is some projection $p\in\CP(\CM)$,
     with $\tau(1-p)<\epsilon$ and $n_0\in \N$ so that
     $\lim_n \|(S-S_n)p\|=0$. We may find $n_0\in \N$ so that for
     $n\geq n_0$, $\|(S-S_n)p\|<\infty$.  As
     $S,S_1,...,S_{n_0}\in \widetilde{\CM}$, there is another projection
     $q$ with  $\tau(1-q)<\epsilon$ and $\|(S-S_k)q\|<\infty$ for $k\leq n_0$.
     Set $p'=p\wedge q$, so that $(S-S_k)p'\in \CM$ for all $k\in \N$ and
     $\lim_n \|(S-S_n)p'\|=0$. By the Cesaro theorem in $\CM$, we get
     $\lim_n \|(S-y_n)p'\|=0$ and $\tau(1-p')\leq 2\epsilon$. This gives the desired a.u. convergence.    
   \end{proof}
   
Let $x \in \widetilde{\CM}$ and $t\in(0,1)$, then the generalised $s$-number $\mu_t(x)$ is defined by
$$
\mu_t(x):= \inf \{ s \ge 0: \tau(1_{(s, \infty)}(\abs{x})) \leq t \}.
$$

We will now discuss the definitions of independent random variables in $\CM$. Let $\CM_1, \CM_2$ be two subalgebras of $\CM$. 
\begin{defn}
    We say that $\CM_1$ and $\CM_2$ are independent if 
    $$
    \tau(xy)= \tau(x) \tau(y)
    $$
    for all $x \in \CM_1$ and $y \in \CM_2$. Moreover, a sequence of random variables $(x_i)_{i \geq 1}$ is said to be independent if for all $i\geq 1$ the von Neumann algebras  $\CM_i:= \{1, x_1,...,x_i \}''$ and ${x_{i+1}}''$ are independent.
\end{defn}

Now let $(\CM_n)_{n \in \N}$ be an increasing sequence of von Neumann subalgebras of $\CM$ such that $\cup_{n \in \N} \CM_n$ is $w^*$-dense in $\CM$. Then for all $n \in \N$ there exists a normal  trace preserving conditional expectation $E_n: \CM \to \CM_n$, satisfying $E_m \circ E_n= E_{\min\{m,n\}}$ for all $m,n \in \N$. A sequence $(x_n)_{n \in \N}$ in $\CM$ is called a martingale sequence with respect to $(\CM_n)_{n \in \N}$ if $E_n(x_{n+1})= x_n$ for all $n \in \N$. For a martingale sequence $(x_n)_{n \in \N}$, we write the martingale differences by $dx_n:= x_n- x_{n-1}$ for all $n \in \N$ (with $x_{-1}=0$). In this setting, without loss of generality, we assume that $\CM_0 = \C$ and $E_0$ denotes the trivial conditional expectation from $\CM$ onto $\CM_0$. Throughout, martingale differences will be denoted by $(dx_n)_{n \in \N}$, which implies $dx_1=x_1$ when $x_0=0$.

Let $1 \leq p< \infty$. The non-commutative $L_p$-space associated to $(\CM, \tau)$ is defined as 
\[L_p(\CM, \tau):= \{x \in \widetilde{\CM}: \tau(\abs{x}^p)< \infty \}
\]
with the Banach space norm $\norm{x}_p:= \tau(\abs{x}^p)^{1/p}$.
The $\ell_\infty$-column valued non-commutative $L_p$-space is defined as 
\[
L_p(\ell_\infty^c):= \{(x_n)_{n \in \N} \subset L_p(\CM, \tau): \text{ there exist } a \in L_{p}(\CM, \tau) \text{ and } (y_n )_{n \in \N} \subset \CM \text{ such that } x_n=  y_n a \}
\]
with associated Banach space norm (when $p\geq 2$) defined by \[
\norm{(x_n)_{n \in \N}}_{L_p(\ell_\infty^c)}:= \inf \{  \sup_{n \in \N} \norm{y_n} \norm{a}_{p} \},
\]
where the infimum is taken over all such factorization as above.

The next statement is a non-commutative, asymmetric analogue of Doob’s maximal inequality, originally established by Junge in \cite{junge2002doob}.

\begin{thm}\label{doob}
   Let $4 \le p \le \infty$. For every $L_p(\mathcal{M})$-bounded martingale
  $x=(x_n)_{n \in \N}$:
 $$\norm{(x_n)_{n \in \N}}_{L_p(\ell_\infty^c)}\leq 2^{2/p}
    \sup_{n\in \N}\|x_n\|.$$    
 \end{thm}

Let us recall the following Golden-Thompson inequalities and prove an exponential inequality which is crucial for our main result.

\begin{lem}[Golden–Thompson Inequality]\cite{ruskai1972inequalities}\label{lem: GT ineq}
    Let $a$ and $b$ be self-adjoint operators that are bounded above, and assume $a + b$ is essentially self-adjoint (that is, its closure is self-adjoint). Then the following inequality holds:

    \begin{equation}
        \tau(e^{a + b}) \leq \tau(e^{a/2} e^b e^{a/2}).
    \end{equation}
    
    \noindent Moreover, if either $\tau(e^a) < \infty$ or $\tau(e^b) < \infty$, then:
    \begin{equation}\label{GT ineq}
        \tau(e^{a + b}) \leq \tau(e^a e^b).
    \end{equation}
\end{lem}

\begin{lem}[Improved Golden–Thompson inequality]\cite{randrianantoanina2024triple}\label{lem: improved GT ineq}
    Let $\mathcal{M}$ be a finite von Neumann algebra equipped with a tracial state $\tau$ and $1 \leq p, q \leq \infty$ with $\frac{1}{p} + \frac{1}{q} = 1$. Assume that $a$, $b$, and $c$ are self-adjoint $\tau$-measurable operators satisfying: $a \in \mathcal{M}$, $e^b \in L_p(\mathcal{M}, \tau)$, and $e^c \in L_q(\mathcal{M}, \tau)$. The following inequality holds:
    $$
    \tau(e^{a + b + c}) \leq \int_0^\infty \tau\left( e^{c/2}(e^{-a} + t1)^{-1} e^b (e^{-a} + t1)^{-1} e^{c/2} \right) dt.
    $$
\end{lem}

The proof of the next result is  very close to \cite[Theorem 4.3]{randrianantoanina2024triple}. By introducing a subtle modification to their argument, we are able to derive a strengthened form of the corresponding inequality. This enhancement plays a central role in establishing our improved version of the law of the iterated logarithm, serving as a key technical component in the overall argument. In addition, the refined bound obtained here may be of independent interest for related problems in the non-commutative probability framework.

\begin{thm}\label{thm: exp ineq}
    Let $ (x_k)_{k=0}^n $ be a finite self-adjoint martingale sequence with respect to the filtration $ (\mathcal{M}_k, E_k) $, and let $ d_k := dx_k = x_k - x_{k-1} $ represent the associated martingale differences.  Let the following conditions hold:
\begin{itemize}
    \item[(i)] $x_0 = 0 $, 
    \item[(ii)] $ \|d_k\| \leq M $ for all $1 \leq k \leq n$,
   \item[(iii)] $ \sum_{k=1}^n E_{k-1}(d_k^2) \leq D^2 1 $.
\end{itemize}
 Then, the following hold:
\begin{enumerate}
    \item For all $\lambda>0$, 
    \[
    \tau\left( e^{\lambda x_n} \right) \leq \exp \left[ \frac{e^{\lambda M} - \lambda M -1}{M^2} D^2 \right]
    \]
    \item For any $ \epsilon \in (0, 1] $ and $ \lambda \in \left[ 0, \frac{3 \epsilon}{M} \right] $,
    \[
    \tau\left( e^{\lambda x_n} \right) \leq \exp\left[ \frac{1}{2} (1 + \epsilon) \lambda^2 D^2 \right].
    \]
\end{enumerate}

\end{thm}

\begin{proof}
\emph{(1):} For $s \in \R$, let us first define the function 
\[
F(s):= e^s-s-1.
\] 
Note that $F(s) \leq F(\abs{s})$.
Now for $n \geq 1$ and $\lambda>0$, let us define for $1\leq l\leq n$
\[
z_l:= \frac{F(\lambda M)}{M^2} \sum_{k=1}^l E_{k-1} (d_k^2).
\]
Then, for any $n \in \N$ and $\lambda>0$, by \Cref{lem: GT ineq}, we have 
\begin{align*}
    \tau(e^{\lambda x_n}) \leq \tau(e^{\lambda x_n - z_n} e^{z_n}).
\end{align*}
Since for all $n \in \N$, $0\le z_n \leq \frac{F(\lambda M)}{M^2} D^2\cdot 1$, we further obtain
\[
\tau(e^{\lambda x_n}) \leq \exp \left[\frac{F(\lambda M)}{M^2} D^2 \right] \tau(e^{\lambda x_n - z_n}).
\]
We claim that 
\begin{equation}\label{eq: claim}
    \tau\left(e^{\lambda x_l - z_l}\right) \leq 1, \text{ for every } n\geq l\ge 1.
\end{equation}

 Indeed, fix $l\ge 1$ and write,
\begin{equation}
    \lambda x_l-z_l = \alpha_l +\beta_l+\eta_l,
\end{equation}

where,
$$\alpha_l:=-\frac{F(\lambda M)}{M^2} E_{l-1}\left(d_l^2\right), ~\beta_l:= \lambda  d_l,~\text{ and } \eta_l:=\lambda x_{l-1}-z_{l-1}.$$
Notice that, $\alpha_l,\beta_l$ and $\eta_l$ are all bounded self-adjoint operators and $z_0=0$. Using \Cref{lem: improved GT ineq}, we have:
\begin{align*}
    \tau\left(\exp\left(\lambda x_l-z_l\right)\right)&= \tau\left(e^{\alpha_l+\beta_l+\eta_l}\right)\\
    &\le \int_0^{\infty} \tau\left(e^{\eta_l}(e^{-\alpha_l} + t \mathbf{1})^{-1} e^{\beta_l}(e^{-\alpha_l} + t \mathbf{1})^{-1}\right) dt.
\end{align*}

Given that both $\alpha_l$ and $\eta_l$ lie in $\CM_{l-1}$, the trace-preserving property of the conditional expectation $E_{l-1}$ allows us to deduce that
\begin{align}
    \tau\left(\exp\left(\lambda x_l-z_l\right)\right)\le \int_0^{\infty} \tau\left(e^{\eta_l} (e^{-\alpha_l} + t\mathbf{1})^{-1} E_{l-1}(e^{\beta_l})(e^{-\alpha_l} + t\mathbf{1})^{-1} \right) dt. \label{eq: tau invariance ineq}
\end{align}

Consider a real valued function $g$ on $\R_+$ defined by:
$$g(s):=\frac{F(s)}{s^2}.$$

Now recall that, for $s\in \R_+,$ we have $s\le e^{s-1}.$ Using functional calculus, it yields that:
\begin{align*}
    E_{l-1}(e^{\beta_l}) =E_{l-1}(e^{\lambda d_l}) &\leq \exp\left( E_{    l-1}\left(e^{\lambda d_l} - 1\right) \right)\stepcounter{equation}\tag{\theequation}\label{eq: expectation ineq}\\
    &= \exp\left( E_{l-1}\left(e^{\lambda d_l} - \lambda d_l - 1\right) \right)\\
    &\le\exp\left( E_{l-1}\left(F\left(\abs{\lambda d_l}\right)\right) \right) \\
    &= \exp\left( E_{l-1}\left(\lambda^2d_l^2 g(\lambda \abs{d_l}) \right) \right).
\end{align*}
Since the function $g(s)$ is strictly increasing on the interval $(0, \infty)$, we must have $g(\lambda \abs{d_l}) \leq g(\lambda M)$. Therefore, 
\begin{equation}\label{eq: function ineq}
     E_{l-1}\left(\lambda^2d_l^2g\left(\lambda \abs{d_l}\right)\right)\le \frac{F(\lambda M)}{M^2}E_{l-1}\left(d_l^2\right).
\end{equation}

Now combining \cref{eq: expectation ineq} and \cref{eq: function ineq}, we have:
\begin{equation}\label{eq: alpha_n ineq}
    E_{l-1}\left(e^{\beta_l}\right)\le \exp\left(\frac{F(\lambda M)}{M^2}E_{l-1}\left(d_l^2\right)\right)=\exp\left(-\alpha_l\right).
\end{equation}

By applying both \cref{eq: tau invariance ineq} and \cref{eq: alpha_n ineq}, we can deduce the following estimate:

$$
\tau\left( \exp(\lambda x_l - z_l) \right) \leq \int_0^{\infty} \tau\left( e^{\eta_l} (e^{-\alpha_l} + t\mathbf{1})^{-1} e^{-\alpha_l} (e^{-\alpha_l} + t\mathbf{1})^{-1} \right) dt.
$$

Now since for all $s>0$
$$
\int_0^\infty \frac{s}{(s + t)^2} \, dt = 1,
$$ 
using functional calculus, we obtain 
\[
\int_0^{\infty}   (e^{-\alpha_l} + t\mathbf{1})^{-1} e^{-\alpha_l} (e^{-\alpha_l} + t\mathbf{1})^{-1}  dt=1.
\]

Therefore, for all integers $l$ with  $n\geq l\geq 1$, we obtain
$$
\tau\left(\exp(\lambda x_l - z_l)\right) \leq \tau(e^{\eta_l})= \tau\left(\exp(\lambda x_{l-1} - z_{l-1})\right).
$$

Hence, by induction, we get
$$
\tau\left(\exp(\lambda x_n - z_n)\right) \leq \tau\left(\exp(\lambda x_0 - z_0)\right)=1.
$$
This proves the claim.

    \emph{(2):} Observe that, for every $\epsilon>0$ whenever $\lambda \in \left[ 0, \frac{3 \epsilon}{M} \right]$, we have    
        $$e^{\lambda M} - \lambda M -1 \leq (1+ \epsilon) \frac{\lambda^2 M^2}{2}.$$    
    Therefore, the result follows.
\end{proof}

An essential component of the arguments involves a standard extension of the classical Borel–Cantelli lemma. To describe this,  which introduces the column-wise tail probability for a sequence of random variables $(x_i)_{i \in I}$. It is given by:

$$
\operatorname{Prob}_c\left( \sup_{i \in I} \|x_i\| > t \right) 
= \inf \left\{ s > 0 : \exists \text{ projection } e \text{ with } \tau(1 - e) < s \text{ and } \|x_i e\|_\infty \leq t \ \forall i \in I \right\},
$$

for any threshold $t > 0$. This definition directly implies the following monotonicity property:

$$
\operatorname{Prob}_c\left( \sup_{i \in I} \|x_i\| > t \right) 
\leq 
\operatorname{Prob}_c\left( \sup_{i \in I} \|x_i\| > r \right),
$$

whenever $t \geq r$. Moreover, if each coefficient $a_i \geq 1$ for $i \in I$, then scaling the sequence yields:

\begin{equation}\label{scaling}
\operatorname{Prob}_c\left( \sup_{i \in I} \|x_i\| > t \right) 
\leq 
\operatorname{Prob}_c\left( \sup_{i \in I} \|a_i x_i\| > t \right).
\end{equation}

For more about the column-wise tail probability, we refer to \cite{konwerska2008law} and  \cite{zeng2015kolmogorov}.

\begin{lem}[non-commutative Borel--Cantelli lemma]\cite[Lemma 3.17]{konwerska2008law}\label{lem: Borel-Cantelli lemma}
    Let $(k_n)_{n \geq n_0}$ be an increasing sequence of integers and $(z_n)$ be a sequence of random variables. Then 
    $$
    \forall \delta>0, \quad \sum_{n \geq n_0} \operatorname{Prob}_c\left( \sup_{k_n\leq m< k_{n+1}} \|z_m\| > \gamma + \delta \right) < \infty
\quad \Longrightarrow\quad \limsup_{n \to \infty} z_n \underset{a.u.}{\leq } \gamma.
    $$
\end{lem}

\begin{lem}[non-commutative Chebyshev inequality]\cite[Lemma 3.16]{konwerska2008law}\label{lem: Chebyshev inequality lemma}
    Let $(x_i)_{i \in I}$ be a sequence of random variables. For $t > 0$ and $1 \leq p < \infty$,
    $$
    \operatorname{Prob}_c\left( \sup_n \|x_n\| > t \right) \leq t^{-p} \|x\|_{L_p(\ell_\infty^c)}^p.
    $$
\end{lem}

\section{LIL for Martingales}\label{sec: main sec}

This section is devoted to the proof of a non-commutative martingale analogue of the law of iterated logarithm. We begin with the precise description. 

Let $(x_n)_{n \in \N}$ be a sequence of martingales with respect to a filtration $(\CM_n)_{n \in \N}$. Furthermore, recall that for all $n \in \N$, let ${\color{blue}d_n:=} dx_n:= x_n - x_{n-1}$ denote the martingale differences. Then we will first prove the following theorem.

\begin{thm}\label{thm: main thm}
    Let $0=x_0, x_1, x_2, \ldots$ be a sequence of self-adjoint martingales in the von Neumann algebra $(\mathcal{M}, \tau)$. Let $s_n^2=\norm{\sum_{k=1}^n E_{k-1}(d_k^2)}$ and $u_n=L(s_n^2)^{1/2}$. Suppose that $s_n^2 \to \infty$ and $\norm{d_n} \leq \frac{\alpha_n s_n}{u_n}$, where $(\alpha_n)$ is a sequence of positive {\color{blue} real} numbers such that $\alpha_n \to 0$, then 
    \[
    \limsup_{n \to \infty} \frac{x_n}{s_n u_n} \underset{a.u.}{\leq }\sqrt{2}.
    \]
\end{thm}

In \cite{zeng2015kolmogorov}, the author proved this theorem with the bound of the limit 
supremum equal to $2$. To establish this result, the martingale inequality 
\cite[Lemma~3]{zeng2015kolmogorov} was employed. However, the application of the aforesaid 
lemma in the proof of \cite[Theorem~1]{zeng2015kolmogorov} is  not totally correct, since for the chosen 
real sequence $(\alpha_n)$, the constant 
$M = \alpha_{k_{n+1}} \frac{s_{k_{n+1}}}{u_{k_{n+1}}}
$
does not necessarily bound $\|d_k\|$ for all $1 \leq k \leq n$. Here, we rectify this inaccuracy by constructing another real sequence $(\alpha_n')$, 
associated with the given sequence $(\alpha_n)$ in the following proposition, which satisfies 
the aforementioned property. In the process, we also improve \cite[Lemma~3]{zeng2015kolmogorov} 
(cf.~\Cref{thm: exp ineq}), and by using this improved version, we strengthen the bound to 
$\sqrt{2}$, which is consistent with the classical result of \cite{stout1970hartman}.

\begin{prop}\label{prop: seq lemma}
    Let $(\alpha_n)$ and $(\beta_n)$ be sequences of positive real numbers such that $\alpha_n \to 0$ and $\beta_n \to \infty$, with $(\beta_n)$ being monotonically increasing. Then there exists a sequence $(\alpha_n')$ of positive real numbers satisfying the following properties:
\begin{enumerate}
    \item $\alpha_n' \geq \alpha_n$ for all $n \in \mathbb{N}$,
    \item $\alpha_n' \downarrow 0$, that is, $(\alpha_n')$ is decreasing and converges to zero,
    \item the sequence $(\alpha_n' \beta_n)$ is non-decreasing; that is, $\alpha_n' \beta_n \leq \alpha_{n+1}' \beta_{n+1}$ for all $n \in \mathbb{N}$.
\end{enumerate}
\end{prop}

\begin{proof}
    By replacing the original sequence $(\alpha_n)$ with the sequence $\tilde{\alpha}_n := \sup_{j \geq n} \alpha_j$, we may assume without loss of generality that $(\alpha_n)$ is decreasing and tends to zero. Since $(\beta_n)$ is an increasing sequence diverging to infinity as $n \to \infty$, we can express it in the form  
    $  \beta_n = \prod_{k=1}^n \gamma_k,
    $  
    for some sequence $(\gamma_k)$ of real numbers satisfying $\gamma_k \geq 1$ for all $k$. Moreover, the fact that  
    $\lim_{n \to \infty} \beta_n = \infty
    $  
    guarantees the existence of an increasing sequence $(n_k)$ such that 
    $   \lim_{k \to \infty} \prod_{i=n_k+1}^{n_{k+1}} \gamma_i = \infty$.

   We define the sequence $(\alpha_n')$ inductively in the following way:  
    \[
    \alpha_1':= \alpha_1;~ \alpha_{n+1}'= \max \left\{\alpha_{n+1}, \frac{\alpha_{n}'}{\gamma_{n+1}}\right\}~ \text{ for } n \geq 2.
    \]
    
    First observe that $\alpha_n \leq \alpha_n'$ for all $n \in \N$. Since for each $n \in \N$, we have $\alpha_{n+1} \leq \alpha_n \leq \alpha_n'$ and $\frac{\alpha_{n}'}{\gamma_{n+1}} \leq \alpha_n'$, we have $\alpha_{n+1}' \leq \alpha_n'$. Moreover, for each $n \in \N$
    \[
    \alpha_{n+1}' \gamma_{n+1}= \max \left\{ \alpha_{n+1} \gamma_{n+1}, \alpha_n' \right\} \geq \alpha_n'.
    \]
    Hence, $(3)$ holds true.
    Let $k \in \N$ be fixed. Consider the numbers $\{ \alpha'_{n_k}, \ldots, \alpha'_{n_{k+1}}  \}$. Suppose, there exists $n_k \leq j \leq n_{k+1}$ such that $\alpha_j'= \alpha_{j}$. In this case, since both $(\alpha_n)$ and $(\alpha_n')$ are decreasing sequences, we obtain $\alpha_{n_{k+1}}' \leq \alpha_j \leq \alpha_{n_k}$. Otherwise, if for all $n_k \leq j \leq n_{k+1}$, $\alpha_j'= \frac{\alpha_{j-1}'}{\gamma_j}$, then 
    \[
    \alpha_{n_{k+1}}'= \frac{\alpha_{n_k}'}{\prod_{i=n_k+1}^{n_{k+1}} \gamma_i} \leq \frac{\alpha_1'}{\prod_{i=n_k+1}^{n_{k+1}} \gamma_i}.
    \]

    Therefore, we proved that 
    \[
    \alpha_{n_{k+1}}' \leq \max \left\{ \alpha_{n_k}, \frac{\alpha_1'}{\prod_{i=n_k+1}^{n_{k+1}} \gamma_i} \right\}~ \text{for all } k \in \N.
    \]

    This implies that $\lim_{k \to \infty} \alpha_{n_{k+1}}'=0$, and since $(\alpha_n')$ is a decreasing sequence, we have $\lim_{n \to \infty} \alpha_n'=0$.
\end{proof}

Now we are ready to prove the main result of this section. The proof basically follows that of Zeng \cite{zeng2015kolmogorov} optimizing  the constant thanks to Theorem \ref{thm: exp ineq} and correcting the inaccuracy with Proposition \ref{prop: seq lemma}. 
\begin{proof}
\emph{of \Cref{thm: main thm}:}
  By the non-commutative Borel–Cantelli lemma, we aim to show that for any $\delta'>0$, 
 $$\sum_n \operatorname{Prob}_c\left( \sup_{k_n < m \le k_{n+1}} \left\| \frac{x_m}{s_m u_m} \right\| > \sqrt 2 (1+ \delta') \right)<\infty.$$
  
  We can find constants $\delta , \epsilon' > 0$,  $\eta \in (1,2)$ and $\epsilon\in (0,1)$ such that
\begin{equation}\label{epscond}
  1 + \delta' > \eta(1 + \delta)(1 - \epsilon')^{-1}, \qquad (1 + \delta)^2 / (1 + \epsilon) > 1. 
\end{equation}

    Set $k_0 = 0$, and for each $n \geq 1$, we define:
    $$
    k_n = \inf \left\{ j \in \mathbb{N} : s_{j+1}^2 \ge \eta^{2n} \right\}.
    $$
    Then $s^2_{k_n+1} \ge \eta^{2n}$ and $s^2_{k_n} < \eta^{2n}$. Now notice that there exists $N_1(\epsilon') > 0$ such that for $n > N_1(\epsilon')$,
    $$
    \frac{s^2_{k_n+1}u^2_{k_n+1}}{s^2_{k_{n+1}}u^2_{k_{n+1}}} \ge \eta^{-2} \frac{\ln \ln \eta^{2n}}{\ln \ln \eta^{2(n+1)}} \ge \left(1 - \epsilon'\right)^2 \eta^{-2}.
    $$
    Therefore, $s_m u_m \ge (1 - \epsilon')\eta^{-1}s_{k_{n+1}}u_{k_{n+1}}$ for $k_n < m \le k_{n+1}$. Furthermore, for all $n > N_1(\epsilon')$ and $\lambda>0$, we have using \eqref{scaling}
    \begin{equation}\label{eq: x_m and x_n ineq}
        \operatorname{Prob}_c\left( \sup_{k_n < m \le k_{n+1}} \left\| \frac{x_m}{s_m u_m} \right\| > \sqrt2(1 + \delta') \right)
    \le \operatorname{Prob}_c\left( \sup_{k_n < m \le k_{n+1}} \left\| \frac{\lambda x_m}{s_{k_{n+1}} u_{k_{n+1}}} \right\| > \lambda \sqrt 2(1 + \delta) \right).
    \end{equation}
    
    Notice that for $p \ge 4$, by Lemma \ref{lem: Chebyshev inequality lemma} and Theorem \ref{doob}
    \begin{align*}
    \operatorname{Prob}_c\left( \sup_{k_n < m \le k_{n+1}} \left\| \frac{\lambda x_m}{s_{k_{n+1}} u_{k_{n+1}}} \right\| > \lambda \sqrt2(1 + \delta) \right)
    &\le \left( \lambda \sqrt2(1 + \delta) \right)^{-p} \left\| \left( \frac{\lambda x_m}{s_{k_{n+1}} u_{k_{n+1}}} \right)_{k_n < m \le k_{n+1}} \right\|^p_{L_p(\ell_\infty^c)}\\
    &\le \left( \lambda \sqrt2(1 + \delta) \right)^{-p} \cdot \left(2^\frac{2}{p}\right)^p \left\| \frac{\lambda x_{k_{n+1}}}{s_{k_{n+1}} u_{k_{n+1}}} \right\|^p_p.
    \end{align*}
    
    Then using functional calculus for the inequality $|u|^p \le p^p e^{-p} (e^u + e^{-u})$, $u \in \R$, we find $$
    \left\| \frac{\lambda x_{k_{n+1}}}{s_{k_{n+1}} u_{k_{n+1}}} \right\|^p_p
    \le p^p e^{-p} \tau\left( \exp\left( \frac{\lambda x_{k_{n+1}}}{s_{k_{n+1}} u_{k_{n+1}}} \right) + \exp\left( -\frac{\lambda x_{k_{n+1}}}{s_{k_{n+1}} u_{k_{n+1}}} \right) \right).
    $$
Next we use \Cref{thm: exp ineq} to the martingale $\left( \frac
{x_ m}{s_{k_{n+1}} u_{k_{n+1}}}\right)_{m=0}^{k_{n+1}}$ to get an upper bound. To check assumption ii), consider the increasing sequence  $\left( \frac{s_n}{u_n} \right).$ Since the sequence $(\alpha_n)$ converges to $0$, by applying \Cref{prop: seq lemma}, we can construct a decreasing sequence $(\alpha_n')$ such that  $\alpha_n' \to 0$
    and the sequence  
    $\left( \frac{\alpha_n' s_n}{u_n} \right)_n$ 
    is increasing. We have
    $\|d_l\|\leq \alpha_l'\frac {s_l}{u_l}\leq \alpha_{k_{n+1}}'\frac {s_{k_{n+1}}}{u_{k_{n+1}}}$ for $1\leq l\leq k_{n+1}$.
    Thus assumption ii) holds with $M=\frac{\alpha_{k_{n+1}}'}{u_{k_{n+1}}^2}$ and for iii) we take $D^2 = 1/u^2_{k_{n+1}}$. Thus for  $n \ge 2$ and $0 \le \lambda \le \frac{3\epsilon u_{k_{n+1}}^2}{\alpha'_{k_{n+1}}} $, we find
    \begin{equation}\label{alpha1}
    \left\| \frac{\lambda x_{k_{n+1}}}{s_{k_{n+1}} u_{k_{n+1}}} \right\|^p_p
    \le 2 \left( \frac{p}{e} \right)^p \exp\left( \frac{(1 + \epsilon)\lambda^2}{2u^2_{k_{n+1}}} \right).\end{equation}
    Hence, for all $n > N_1(\epsilon') $ and $0 \le \lambda \le \frac{3\epsilon u_{k_{n+1}}^2}{\alpha'_{k_{n+1}}} $
    \begin{equation}\label{eq: Probc ineq}
        \operatorname{Prob}_c\left( \sup_{k_n < m \le k_{n+1}} \left\| \frac{\lambda x_m}{s_{k_{n+1}} u_{k_{n+1}}} \right\| > \lambda \sqrt2(1 + \delta) \right)
        \le 8 \left( \frac{p}{\lambda \sqrt2(1 + \delta) e} \right)^p \exp\left( \frac{(1 + \epsilon)\lambda^2}{2u^2_{k_{n+1}}} \right).
    \end{equation}
    
    Now take $p = \lambda \sqrt2(1 + \delta)$, it is clearly bigger than 4 if $\lambda\geq u_{k_{n+1}}$ and $n>N_2$ for some $N_2$. Thus from \cref{eq: Probc ineq} we have,
    $$
    \operatorname{Prob}_c\left( \sup_{k_n < m \le k_{n+1}} \left\| \frac{\lambda x_m}{s_{k_{n+1}} u_{k_{n+1}}} \right\| > \lambda \sqrt2(1 + \delta) \right)
    \le 8 \exp\left( \frac{(1 + \epsilon)\lambda^2}{2u^2_{k_{n+1}}} - \sqrt2(1 + \delta) \lambda \right).
    $$
    Since $\alpha'_n \to 0$, there exists $N_3 > 0$ such that for $n > N_3$,
    \begin{equation}\label{alpha2}
      0 < \alpha'_{k_{n+1}} \le \frac{3\epsilon (1+\epsilon) }{\sqrt2(1+\delta)} \quad \textrm{and}\quad \frac{\sqrt{2}(1+\delta)}{(1+\epsilon)}u_{k_{n+1}}\geq 1.
    \end{equation}

    Now for $n \geq \max \{ N_1(\epsilon'), N_2, N_3 \}$ 
    set $\lambda = \sqrt2(1 + \delta)u^2_{k_{n+1}} / (1 + \epsilon)\geq u_{k_{n+1}}$ and observe that 
    $$0 \le \lambda \le \frac{3\epsilon u_{k_{n+1}}^2}{\alpha'_{k_{n+1}}}. $$
        Therefore { as $e^{u_{k_{n+1}}^2}=\ln(s_{k_{n+1}}^2)$}, it follows that
    $$
    \operatorname{Prob}_c\left( \sup_{k_n < m \le k_{n+1}} \left\| \frac{\lambda x_m}{s_{k_{n+1}} u_{k_{n+1}}} \right\| > \lambda \sqrt2(1 + \delta) \right)
    \le 8\left( \ln s^2_{k_{n+1}} \right)^{- (1+\delta)^2/(1+\epsilon)}
    $$
    for all $n \ge \max \{N_1(\epsilon'), N_2, N_3 \}$. Notice that for all $n \in \N$,  $s^2_{k_{n+1}} \ge s^2_{k_n + 1} \ge \eta^{2n}$, we have
    $$
    \operatorname{Prob}_c\left( \sup_{k_n < m \le k_{n+1}} \left\| \frac{\lambda x_m}{s_{k_{n+1}} u_{k_{n+1}}} \right\| > \lambda \sqrt{2}(1 + \delta) \right)
    \le \left[ (2 \ln \eta)n \right]^{- (1 + \delta)^2 / (1 + \epsilon)}.
    $$
    
Recall that by \eqref{epscond},  $(1 + \delta)^2 / (1 + \epsilon) > 1$, we find that for $n_0 = \max\{N_1(\epsilon'), N_2, N_3 \}$,
    $$
    \sum_{n \ge n_0} \operatorname{Prob}_c\left( \sup_{k_n < m \le k_{n+1}} \left\| \frac{x_m}{s_{k_{n+1}} u_{k_{n+1}}} \right\| > \sqrt2(1 + \delta) \right) < \infty.
    $$
    Then \Cref{eq: x_m and x_n ineq} and \Cref{lem: Borel-Cantelli lemma} give the desired result.
\end{proof}
We note that the classical martingale analogue of the law of the iterated logarithm 
(cf.~\cite{stout1970martingale}) holds for non-self-adjoint martingales.  In the following, we remove the self-adjointness condition in \Cref{thm: main thm} { by standard tricks.} 
\begin{thm}\label{trick}
    Let $(x_n)_{n \geq 0}$ be a non-commutative martingale in $(\mathcal{M}, \tau)$ with $x_0 = 0$. Define \newline $t_n^2 = \max\left\{\norm{\sum_{i=1}^n E_{i-1}(d_i^*d_i)},\norm{\sum_{i=1}^n E_{i-1}(d_id_i^*)}\right\}$ and $v_n = [L(t_n^2)]^{1/2}$. Suppose that $t_n^2 \to \infty$ and $\norm{d_n}_\infty \leq \frac{\alpha_n t_n}{v_n}$, where $(\alpha_n)$ is a sequence of positive numbers such that $\alpha_n \to 0$, then 
    $$
    \limsup_{n \to \infty} \frac{x_n}{t_n v_n} \underset{a.u.}{\leq }\sqrt{2}.
    $$
\end{thm}
\begin{proof}
    Consider the von Neumann algebras 
    $$
    \tilde{\mathcal{M}} := \mathcal{M} \otimes M_2(\mathbb{C}) \quad \text{and} \quad \tilde{\mathcal{M}}_k := \mathcal{M}_k \otimes M_2(\mathbb{C})
    $$
    for each integer $k \geq 1$. Let $(\tilde{x}_n)_n$ be a sequence of elements in $\tilde{\mathcal{M}}$ defined by
    $$
    \tilde{x}_n := \begin{pmatrix}
    0 & x_n \\
    x_n^* & 0
    \end{pmatrix}, \quad n \geq 0,
    $$
    where $x_n \in \mathcal{M}$.
    
    The sequence $(\tilde{x}_n)_n$ forms a self-adjoint martingale with respect to the increasing filtration $(\tilde{\mathcal{M}}_n, \mathcal{E}_n)$, where the conditional expectation 
    $\mathcal{E}_n : \tilde{\mathcal{M}}_{n+1} \to \tilde{\mathcal{M}}_n$
    satisfies
    $$
    \mathcal{E}_n(\tilde{x}_{n+1}) := \begin{pmatrix}
    0 & E_n(x_{n+1}) \\
    E_n(x_{n+1}^*) & 0
    \end{pmatrix}, \quad n \geq 0.
    $$
    Here, $E_n$ denotes the conditional expectation from $\mathcal{M}_{n+1}$ onto $\mathcal{M}_n$.
    Let us define for $n\geq1$,
    $\tilde{d}_n := \tilde{x}_n - \tilde{x}_{n-1}.$
    Observe that each $\tilde{d}_n$ is self-adjoint for all $n \geq 0$. Moreover,
    \[
        \tilde{d}_n^2 = 
        \begin{pmatrix} 0 & d_n \\ d_n^* & 0 \end{pmatrix} 
        \begin{pmatrix} 0 & d_n \\ d_n^* & 0 \end{pmatrix} =
        \begin{pmatrix} d_n d_n^* & 0 \\ 0 & d_n^* d_n \end{pmatrix}.
    \]
    Define, for each $n \geq 1$,
    \[
        \tilde{t}_n^2 := \left\| \sum_{i=1}^n \CE_{i-1}\left(\tilde{d}_i^2\right) \right\|.
    \]
    It follows that $\tilde{t}_n = t_n$ for all $n \geq 1$. Furthermore, for every $n \geq 1$,
    \[
        \|\tilde{d}_n\|_\infty = \|d_n^* d_n\|^{1/2} = \|d_n\|.
    \]
    Consequently, $(\tilde{x}_n)_n$ constitutes a sequence of self-adjoint martingales in the von Neumann algebra $\tilde{\mathcal{M}}$ such that
    \[
        \lim_{n \to \infty} \tilde{t}_n = \infty
        \quad \text{and} \quad
        \|\tilde{d}_n\|_\infty \leq \frac{\alpha_n \tilde{t}_n}{v_n}, \quad \text{for all } n \geq 1.
    \]
    Applying Theorem~\ref{thm: main thm}, we deduce that
    \[
        \limsup_{n \to \infty} \frac{\tilde{x}_n}{t_n v_n} \underset{a.u.}{\leq} \sqrt{2}.
    \]

    Taking $\tilde r=\begin{pmatrix} 0 & 0 \\ 1 & 0 \end{pmatrix}$ in $\tilde \CM$ we first note that $\tilde{r} \tilde{x}_n = \begin{pmatrix} 0 & 0 \\ 0 & x_n \end{pmatrix}$ and obtain, 
    $$\limsup_{n\to \infty} \frac{\tilde{r}\tilde{x}_n }{t_n v_n} \underset{{a.u.}}{\leq} \sqrt{2}.$$
    Since almost uniform boundedness does not depend on the ambient von Neumann algebra (cf. Remark \ref{rem: indep of ambient vna}), we obtain 
    \[
    \limsup_{n\to \infty} \frac{x_n}{t_n v_n} \underset{{a.u.}}{\leq} \sqrt{2}.
    \]
\end{proof}
We now improve  \Cref{thm: main thm} thanks to a cutting argument. This can be thought of as a non-commutative analogue of \cite[Theorem 3]{stout1970martingale}. 
\begin{thm} 
Let $(x_n)$ be a sequence of self-adjoint martingales with $x_0=0$  in the von Neumann algebra $(\mathcal{M}, \tau)$. Assume that 
$s_n^2=\|\sum_{k=1}^n E_{k-1}(d_k^2)\| \to \infty$ and there is a sequence of positive real numbers $(K_n)$ going to 0 and $\alpha>0$ such that
    \[
     \norm{\sum_{n=1}^\infty \frac{u_n^\alpha}{ s_n^2} \, E_{n-1} \left[ d_n^2 \, \mathbf{1}_{(s_n K_n/u_n,\infty)}(|d_n|)  \right]} \leq C< \infty.
    \]
    with $u_n=L(s_n^2)^{1/2}$. Then
    $$
    \limsup_{n \to \infty} \frac{x_n}{s_n u_n}  \underset{a.u.}{\leq }\sqrt{2}.
    $$
\end{thm}
\begin{proof}
For every $n \in \N$ we write 
\[
d_n= \big(d_n''-E_{n-1}(d_n'')\big) + \big(d_n' - E_{n-1}(d_n')\big):=z_n + z_n' ,
\]
where $d_n':= d_n 1_{[0, K_ns_n/u_n]}(\abs{d_n})$ and $d_n'':= d_n - d_n'=d_n {1}_{(s_n K_n/u_n,\infty)}(|d_n|) $. 

We start by showing that $\Big(\frac{\sum_{k=1} ^nz_k}{s_n u_n}\Big)_n$ goes a.u. to 0. By Lemma \ref{lem: Kronecker2}, it suffices to show that $\Big(\sum_{k=1} ^n\frac{z_k}{s_k u_k}\Big)_n$ converges a.u. But since it is a martingale sequence, it suffices to show that it is bounded in $L_2$ by \cite{cuculescu1971martingales}. We have
$$ \sum_{n=1}^\infty \norm{\frac{z_n}{s_n u_n}}_2^2\leq \sum_{n=1}^\infty \frac 1{s_n^2u_n^2}\norm{{d_n''}}_2^2\leq C  <\infty.$$

    Next, to conclude, we claim that
    \[
    \limsup_{n \to \infty} \frac{\sum_{k=1}^n z_k'}{s_n u_n}  \underset{a.u.}{\leq }\sqrt{2}.
    \]
    Then this, together with Lemma \ref{lem: au convergence and bounded lemma}, proves the result.

      We  check that the martingale $(Z_n=\sum_{k=0}^nz'_k)_n$
      satisfies the hypotheses of \Cref{thm: main thm}.
      
  First, let   $s_n'^2:= \norm{\sum_{i=1}^{n} E_{i-1}(z_i'^2)}$. 
  By \cite{junge2002doob}, the quantity $\norm{\sum_{i=1}^n E_{i-1} (h_i^2)}^{1/2}$ is a norm on finite self-adjoint martingale differences $(h_i)_{i=1}^n$. Thus, 
  $$|s_n -s'_n|^2 \leq   \norm{\sum_{i=1}^{n} E_{i-1}(z_i^2)}\leq \norm{\sum_{i=1}^{n} E_{i-1}(d_i''^2)}.$$
      But since $s_i/u_i^{\alpha/2}<A s_n/u_n^{\alpha/2}$ for $i<n$ and some  $A>1$, our assumption gives
    \[
    \norm{\sum_{i=1}^{n} E_{i-1} (d_i^2 1_{(K_is_i/u_i, \infty)} (\abs{d_i}) )} \leq CA^2 s_n^2/u_n^\alpha.
    \]
    Therefore, we obtain
        \begin{equation*}\label{eq: s_n and s_n dash limit}
        \lim_{n\to \infty}\frac{s_n}{s_n'}=1 
        \qquad \textrm{ and also}\qquad \lim\limits_{n\to \infty}\frac{u_n}{u_n'}=1.
    \end{equation*}
        
    Clearly, by definition, for all $n \in \N$, $\norm{z_n'} \leq 2 K_n \frac{s_n}{u_n}= 2 K_n \frac{s_n u_n'}{s_n'u_n} \cdot \frac{s_n'}{u'_n} = K_n' \frac{s_n'}{u_n'}$ with $K_n'\to 0$. Obviously $z'_0=0$ and we conclude that { 
    \begin{equation*}\label{eq: au less sqrt 2}
        \limsup_{n\to \infty} \frac{\sum_{k=1}^{n}z_k'}{s_n' u_n'}\underset{a.u.}{\leq }\sqrt{2}.
    \end{equation*}The proof of the claim is over as $\lim \frac{u_ns_n}{u'_ns'_n}=1$.}
  \end{proof}

    \section{LIL for independent r.v.}\label{sec: lil for indep rv}

 This section is devoted to the study of the non-commutative LIL for independent random variables. The main result of this section (\Cref{thm: main them with independent rv}) generalises the Hartman-Wintner law of iterated logarithm \cite{hartman1941law} to the non-commutative setting. 
 
 Throughout this section, we fix a von Neumann algebra $\CM$ with f.n. tracial state $\tau$. We say that a random variable $y \in \CM$ to have mean zero if $\tau(y)=0$. For every $n \in \N$, we also set the notation $u_n=L(n)^{1/2}$. Now we have the following crucial result.

\begin{prop}\label{prop : variation}
    For any $\delta'>0$, we can find some $e>0$ such that any sequence $(y_i)$ of  independent mean zero self-adjoint random variables such that 
    \begin{enumerate}
        \item\label{cond1} $\norm{y_i} \leq e \frac{\sqrt{i}}{u_i}$, 
        \item \label{cond2}$\norm{y_i}_2 \leq 1$,
    \end{enumerate}
    we have, with  $x_n=\sum_{k=1}^n y_k$, 
    $$ 
    \limsup_{n \to \infty} \frac{x_n}{\sqrt{n} u_n}  \underset{a.u.}{\leq }\sqrt{2}(1+\delta'). $$
  \end{prop}

  \begin{proof}
   Let $\CM_n=\{1, y_1,...,y_n\}''$. The sequence $(x_n=\sum_{k=1}^n y_k)_{n \ge 0}$, with $x_0 = 0$
  is a self-adjoint martingale in the von Neumann algebra $(\mathcal{M}, \tau)$
  with respect to the filtration $(\CM_n)$. Its difference sequence is given by $d_n:=dx_n=y_n$.  Then, the proof is almost the same as \Cref{thm: main thm} once $\delta'$ is fixed. First, one has $s_n^2\leq n$ by independence and  \eqref{cond2}. Next, we used that the sequence $(\alpha_n')$ goes to 0 twice: 
    \begin{itemize}
    \item right before \eqref{alpha1} to check the assumptions of
      \Cref{thm: exp ineq}; this now is true by our assumption \eqref{cond1}:
      that is we take $M=\frac e{u^2_{k_{n+1}}}$ because
      $l\mapsto \frac {\sqrt l}{u_l}$ is increasing in $l$. 
      \item to have the first inequality in \eqref{alpha2}. Now, this is how we choose $e$, that is so that $0<e<\frac{3\epsilon (1+\epsilon) }{\sqrt2(1+\delta)}$.
      \end{itemize}
  The rest of the arguments are unchanged.     
 \end{proof}

  For $z\in L_1(\CM)$, we denote $\mathring z = z-\tau(z)$. We will need the following estimate:

  \begin{lem}\label{calc}
    Let $e>0$ and $(y_k)_{k\geq 1}$ be a family of independent, identically distributed, mean zero,
    self-adjoint random variables in $L_2(\CM)$, then with $z_k=y_k1_{\sqrt k\geq |y_k|>e \frac{ \sqrt k}{u_k}}$, $\sum_{n=1}^\infty \frac 1 {\sqrt n  {u_n}} \mathring z_n$ converges in $L_2(\CM)$.   \end{lem}

  \begin{proof}
   Let $y(t)=\mu_t(y_1)$ defined on $(0,1)$.
    By orthogonality, we have to show that $\sum_n\frac 1 {n  u_n^2} \|\mathring z_n\|_2^2<\infty$. But by the hypotheses
    $$\sum_n\frac 1 {n  u_n^2} \|\mathring z_n\|_2^2\leq \sum_n\frac 1 {n  u_n^2} \|z_n\|_2^2\leq \sum_n\frac 1 {n  u_n^2} \int_{\sqrt n \geq y>e \frac{ \sqrt n}{u_n}} y^2= \int G(y)y^2,$$
    where for $t>0$,
    $G(t)=\sum_{ \sqrt n \geq t> e \frac{ \sqrt n}{u_n}}\frac 1 {n
      u_n^2}$. For every $t>0$, let $n_t$ be the biggest integer so
    that $u_{n_t}^2 t^2 \geq e^2n_t$. Note that one can easily find
    $a>1$ and $b>0$ (depending on $e$) so that $a \ln(\lceil t^2\rceil )+b \geq \ln(n_t)$ when
    $t>0$. It follows that $u_{n_t}= L(n_t)^{1/2}\leq c u_{\lceil t^2\rceil}$
    for some $c>1$.    We have the estimate
    $$G(t)\leq \sum_{n=\lceil t^2\rceil}^{n_t} \frac 1 {nu_n^2}\leq
    \frac{n_t}{\lceil t^2\rceil u_{\lceil t^2\rceil}^2}\leq \frac{u_{n_t}^2 \lceil t^2\rceil}{e^2\lceil t^2\rceil u_{\lceil t^2\rceil}^2}\leq\frac {c^2}{e^2}. $$
    This means that $G$ is bounded and this concludes the proof.    
  \end{proof}

  \begin{lem}\label{tail}
   Let $e>0$ and $(y_i)_{i\geq 1}$ be a family of independent, identically distributed mean zero,
    self-adjoint random variables in $L_2(\CM)$, then with $w_k=y_k1_{ |y_k|>\sqrt k}$, $\frac 1{\sqrt n}\sum_{k=1}^n \mathring w_k$ goes to 0 a.u.
      \end{lem}
      \begin{proof}
        Let $p_k=1_{ |y_k|>\sqrt k}$ and $q_k=1_{y>\sqrt k}$ with $y=\mu(y_1)$, we have
        $\tau(p_k)= \int q_k$. But $\sum_{k=1}^\infty \int q_k\leq \int y^2<\infty$. Hence for any $\epsilon>0$, we may find $k_0$ so that
        $p'=\wedge_{k\geq k_0}(1-p_k)$ satisfies $\tau(1-p')\leq \epsilon/2$. Taking $p=p'\wedge q$ where $q$ is such that $\tau(1-q)<\epsilon/2$ and $w_kq\in \CM$ for $k< k_0$. We have $\tau(1-p)\leq \epsilon$ and $\frac 1 {\sqrt n}\sum_{k=1}^n w_kp=\frac 1 {\sqrt n}\sum_{k=1}^{k_0} w_kp$, this goes to 0 in $\CM$.

        As $\mathring{w_k}=w_k-\tau (w_k)$. It remains to check that $\frac 1 {\sqrt n}\sum_{k=1}^n \tau(w_k)$ goes to 0. By  Lemma \ref{lem: Kronecker2}, it suffices to prove that $\sum_{n\geq 1} \frac 1 {\sqrt n}\tau(|w_n|)<\infty$. Since $\sum_{k=1}^{\lfloor { y^2}\rfloor} \frac 1{\sqrt k}\leq { 2}|y|$, one gets $\sum_{n\geq 1} \frac 1 {\sqrt n}\tau(|w_n|)<{2}\int |y|^2$.     
           \end{proof}

      \begin{thm}\label{thm: main them with independent rv}
  Let $(y_i)$ be a sequence of independent, identically distributed, mean zero 
  random variables with $\|y_i\|_2\leq 1$, then 
    $$ 
    \limsup_{n \to \infty} \frac{x_n}{\sqrt{n} u_n}  \underset{a.u.}{\leq }\sqrt{2}.
    $$
  \end{thm}

\begin{proof}
We will assume that $y_k$’s are self-adjoint, the general case can be deduced from the proof below using the trick of Theorem \ref{trick}. We want to show that for any $\delta'>0$, 
 $\limsup_{n \to \infty} \frac{x_n}{\sqrt{n}u_n } \underset{a.u.}{\leq }\sqrt{2}(1+\delta')$. 

 Let $e\in (0,1)$ be given by Proposition \ref{prop : variation}. We decompose the martingale into three parts
\begin{equation}\label{decomp}y_k= \mathring{\overbrace{y_k 1_{[0, \frac {e \sqrt k}{2u_k}]}(|y_k|) }}+\mathring{\overbrace{y_k 1_{(\frac {e \sqrt k}{2u_k}, \sqrt k]}(|y_k|) }}+ \mathring{\overbrace{y_k 1_{(\sqrt k,\infty)}(|y_k|) }}= \mathring{y_k'} +\mathring{z_k} + \mathring{w_k}.\end{equation}

 The sequence $(\mathring{y_k'})$ satisfies all the hypotheses in Proposition
 \ref{prop : variation}, so we get, with $x'_n=\sum_{k=1}^n \mathring{y_k'}$: $$ 
 \limsup_{n \to \infty} \frac{x_n'}{\sqrt{n} u_n}  \underset{a.u.}{\leq }\sqrt{2}(1+\delta').$$ 
 Next, by  Lemma \ref{calc}, we have that $\sum_{n=1}^\infty \frac 1 {\sqrt n  {u_n}} \mathring z_n\in L_2$. Thus, the associated martingale sequence
 $(\sum_{k=1}^n \frac 1 {\sqrt k {u_k}} \mathring z_k)_n$ converges a.u. by \cite{cuculescu1971martingales}. Then by Lemma \ref{lem: Kronecker2}, we conclude that  $\frac 1 {\sqrt n  {u_n}}\sum_{k=1}^n
 \mathring z_k$ converges a.u. to 0. Finally, Lemma \ref{tail} yields that 
$\frac 1 {\sqrt n  {u_n}}\sum_{k=1}^n
\mathring w_k$ converges a.u. to 0. One concludes with Lemma
\ref{lem: au convergence and bounded lemma}.
\end{proof}
\begin{rem}
The martingale with differences $\mathring{z_k} + \mathring{w_k}$ is easily seen to be in $L_1$ (see Lemma 2.3 in \cite{deAcosta}). This is a way to get a quick proof for the b.a.u. convergence or the commutative result. It seems simpler than in \cite{deAcosta} as we do not use any maximal inequality that seems unavailable here.

It is also possible to relax the identically distributed assumption to a weaker domination property in distribution that is commonly used (see \cite{hartman1941law} and \cite{konwerska2008law}). We mean that $y=\sup_i \mu_t(y_i)\in L_2(0,1)$. We briefly explain how. First, as a.u. convergence does not depend on the algebra, we can enlarge all the algebras so that $y_k$ sits in a diffuse algebra that is independent from the previous ones. Then, one has to change the decomposition \eqref{decomp}, taking $y'_k=y_k a_k$, $z_k=y_k b_k$ and $w_k= y_k{c_k}$ in such a way that $a_k, b_k, c_k$ are orthogonal projections summing up to 1 so that $\mu(y_k')\leq \mu\big(y1_{[0,\frac {e \sqrt k}{2u_k}]}(y)\big)$, $\mu(z_k)\leq \mu\big(y 1_{(\frac {e \sqrt k}{2u_k}, \sqrt k]}(y)\big)$, $\mu(w_k)\leq \mu\big(y  1_{(\sqrt k,\infty)}(y)\big)$. This is possible as $\mu(y_k)\leq y=\mu(y)$ and the algebras are diffuse.
\end{rem}

\section{Acknowledgements}
This work was partially carried out during the visits of the first and third authors to the University of Caen, LMNO. They also acknowledge the support of Prof. Pierre Fima and Prof. Issan Patri, provided through the CNRS IEA GAOA grant.


\providecommand{\bysame}{\leavevmode\hbox to3em{\hrulefill}\thinspace}
\providecommand{\MR}{\relax\ifhmode\unskip\space\fi MR }
\providecommand{\MRhref}[2]{%
  \href{http://www.ams.org/mathscinet-getitem?mr=#1}{#2}
}

\end{document}